\documentclass[11pt,a4paper]{article}

\usepackage{amsmath,amsfonts,amssymb,geometry,times}
\usepackage[T1]{fontenc}
\usepackage{tikz}
\usepackage{a4wide}
\usepackage{amssymb,amsmath,array,amsthm,amsfonts,amscd}
\usepackage{version}
\usepackage{algpseudocode}
\usepackage{tikz}
\usepackage{mathrsfs}
\usepackage{enumitem}
\usepackage{url}

\allowdisplaybreaks[1]

\usetikzlibrary{decorations.pathmorphing,shapes}

\newcounter{sarrow}

\newtheorem{thm}{Theorem}[section]
\newtheorem{lem}[thm]{Lemma}
\newtheorem{prop}[thm]{Proposition}
\newtheorem{cor}[thm]{Corollary}

\theoremstyle{definition}

\newtheorem{notation}[thm]{Notation}
\newtheorem{ques}[thm]{Question}
\newtheorem{Def}[thm]{Definition}
\newtheorem{hypo}[thm]{Hypothesis}

\theoremstyle{remark}
\newtheorem{rem}[thm]{Remark}

\newtheorem{ex}[thm]{Example}

\newcommand{\titre}{Poisson Deleting Derivations Algorithm and Poisson Spectrum}

\newcommand{\Z}{\mathbb{Z}}
\newcommand{\Q}{\mathbb{Q}}
\newcommand{\C}{\mathbb{C}}
\newcommand{\K}{\mathbb{K}}

\newcommand{\al}{\alpha}
\newcommand{\be}{\beta}

\newcommand{\de}{\delta}
\newcommand{\la}{\lambda}
\newcommand{\bo}{\boldsymbol}
\newcommand{\F}{\mbox{\rm Frac\,}}

\newcommand{\Sp}{\mbox{\rm Spec\,}}
\newcommand{\PS}{\mbox{\rm P.Spec\,}}
\newcommand{\tPS}{\footnotesize\mbox{\rm P.Spec\,}}
\newcommand{\car}{\mbox{\rm char\,}}
\newcommand{\id}{\mbox{\rm id}}
\newcommand{\D}{\mbox{\rm Der}}
\newcommand{\Imm}{\mbox{\rm Im}}
\newcommand{\ov}{\overline}

\begin{document}

\title{\titre} \author{St\'ephane Launois and C\'esar Lecoutre\thanks{The second author thanks EPSRC for its support.}}

\date{}

\maketitle

\begin{center}
\textit{School of Mathematics, Statistics and Actuarial Science (SMSAS), Cornwallis Building, University of Kent, Canterbury, Kent CT2 7NF, United Kingdom}
\end{center}
\begin{center}
{\tt S.Launois@kent.ac.uk} and {\tt C.Lecoutre@kent.ac.uk}
\end{center}


\begin{abstract}
	In \cite{Cau1} Cauchon introduced the so-called deleting derivations algorithm. This algorithm was first used in noncommutative algebra to prove catenarity in generic quantum matrices, and then to show that torus-invariant primes in these algebras are generated by quantum minors. Since then this algorithm has been used in various contexts. In particular, the matrix version makes a bridge between torus-invariant primes in generic quantum matrices, torus-orbits of symplectic leaves in matrix Poisson varieties and totally nonnegative cells in totally nonnegative matrix varieties \cite{GLL}. This led to recent progress in the study of totally nonnegative matrices such as new recognition tests, see for instance \cite{Lalen}. The aim of this article is to develop a Poisson version of the deleting derivations algorithm to study the Poisson spectra of the members of a class $\mathcal{P}$ of polynomial Poisson algebras. It has recently been shown that the Poisson Dixmier-Moeglin equivalence does not hold for all polynomial Poisson algebras \cite{BLSM}. Our algorithm allows us to prove this equivalence for a significant class of Poisson algebras, when the base field is of characteristic zero. Finally, using our deleting derivations algorithm, we compare topologically spectra of quantum matrices with Poisson spectra of matrix Poisson varieties.
\end{abstract}



\section*{Introduction}

	Poisson algebras have been intensively and widely studied since their first appearance, both on their own and in connection with other areas of mathematics. For instance, we refer to \cite{LGPV} where Poisson structures are studied from the differential geometry point of view, \cite{DR} where links with number theory are made or \cite{Goo2} for the connection with noncommutative algebra; this literature is of course non exhaustive. In this paper we study Poisson spectra of certain Poisson polynomial algebras. Different aspects of this topic have been investigated previously: the Poisson Dixmier-Moeglin equivalence is studied in \cite{BLSM}, \cite{Goo1}, \cite{GL} and \cite{Oh2}, links between Poisson spectra and their quantum analogues are investigated in \cite{GoLet3}, \cite{J1}, \cite{Oh2} and \cite{Yak} and Poisson spectra of Jacobian Poisson structures and generalisations in higher dimensions are studied in \cite{JO2} and \cite{JO1}.

	Inspired by \cite{Cau1}, we develop a method to study the algebras of a class $\mathcal{P}$ of iterated Poisson-Ore extensions over a field $\K$ of arbitrary characteristic. More precisely for $A\in\mathcal{P}$, the (characteristic-free) \textit{Poisson deleting derivations algorithm} consists of performing several explicit changes of variables inside the field of fractions $\F A$ of $A$. At each step of the algorithm we obtain a sequence of $n$ algebraically independent elements of $\F A$, where the integer $n$ corresponds to the number of indeterminates in $A$. The subalgebra of $\F A$ generated by these elements is a Poisson algebra with a "simpler" Poisson bracket than the one obtained at the previous step. Moreover the Poisson algebras corresponding to two consecutive steps, say $C_{j+1}$ and $C_j$, satisfy:
\[C_{j+1}S_j^{-1}=C_jS_j^{-1}\]
for a given multiplicatively closed set $S_j$. After the last step, we get algebraically independent elements $T_1,\dots,T_n$ of $\F A$ such that the algebra $\ov{A}$ they generate is a \textit{Poisson affine space}, i.e. $\ov{A}$ is a polynomial algebra $\K[T_1,\dots,T_n]$ with Poisson bracket on the generators given by $\{T_i,T_j\}=\la_{ij}T_iT_j$ for all $i,j$, where $(\la_{ij})\in\mathcal{M}_n(\K)$ is a skew-symmetric matrix. In particular the algorithm shows that $\F A=\F \ov{A}$ as Poisson algebras. Therefore we retrieve the results of Poisson birational equivalence obtained in \cite{LL} (see also \cite{GL} in characteristic zero), that is the Poisson algebras of the class $\mathcal{P}$ satisfy the quadratic Poisson Gel'fand-Kirillov problem (see \cite{GL} and \cite{LL} for more details).
	
	 For a Poisson algebra $A$ we denoted by $\PS(A)$ the set of prime ideals which are also Poisson ideals. We refer to this set as the \textit{Poisson spectrum} of $A$ (see Remark \ref{pp} at the end of the introduction). The set $\PS(A)$ is equipped with the induced Zariski topology from the spectrum $\Sp(A)$ of $A$. When $A\in\mathcal{P}$, our algorithm allows us to define an embedding $\varphi$ from $\PS(A)$ to $\PS(\ov{A})$ called the \textit{canonical embedding}. This embedding will be our main tool for studying Poisson spectra. One of its important properties is that for $P\in\PS(A)$ we have a Poisson algebra isomorphism:
 	\[\F\Big(\frac{A}{P}\Big)\cong\F\Big(\frac{\ov{A}}{\varphi(P)}\Big).\]
Note that this isomorphism reduces the quadratic Poisson Gel'fand-Kirillov problem for the Poisson prime quotients of $A$ to the quadratic Poisson Gel'fand-Kirillov problem for the Poisson prime quotients of a Poisson affine space. As in the noncommutative case, the canonical embedding leads to a partition of $\PS(A)$ indexed by a subset $W_P'$ of $W:=\mathscr{P}([\mspace{-2 mu} [1,n] \mspace{-2 mu} ])$, the powerset of $[\mspace{-2 mu} [1,n] \mspace{-2 mu} ]:=\{1,\dots,n\}$. More precisely, for $w\in W$, we set:
\begin{equation}
\label{psw}
\PS_w(\ov{A}):=\big\{P\in\PS(\ov{A})\ |\ Q\cap\{T_1,\dots,T_n\}=\{T_i\ |\ i\in w\}\big\},
\end{equation}
where we recall that $T_1,\dots,T_n$ are the generators of the Poisson affine space $\ov{A}$. These sets form a partition of $\PS(\ov{A})$ which induces a partition on $\PS(A)$ as follows:
\[\PS(A)=\bigsqcup_{w\in W'_P}\varphi^{-1}\big(\PS_w(\ov{A})\big),\quad\text{where}\quad W'_P:=\{ w \in W ~|~ \varphi^{-1}\big(\PS_w(\ov{A})\big)\neq\emptyset \}.\]
  This partition of $\PS(A)$ is called the {\em canonical partition}, and the elements of $W'_P$ will be called the {\em Cauchon diagrams associated to A}, or Cauchon diagrams for short. For $w\in W'_P$, the set $\varphi^{-1}\big(\PS_w(\ov{A})\big)$ is called the {\em stratum} associated to $w$. We study the topologico-algebraic properties of those strata in Section \ref{taprop}, our main result being that for $w\in W'_P$ the image of the stratum associated to $w$ is a closed subset of $\PS_w(\ov{A})$ and that $\varphi$ induces a homeomorphism from this stratum to its image. In Section \ref{PDM} we turn our attention to Poisson primitive spectra of the algebras of the class $\mathcal{P}$. In particular our algorithm allows us to prove the Poisson Dixmier-Moeglin equivalence for the algebras of the class $\mathcal{P}$ when $\car\K=0$. For information on the original Dixmier-Moeglin equivalence, as well as its Poisson version we refer to \cite{BLSM} and \cite{Goo1}. We briefly recall here the Poisson version. Let $A$ be a Poisson $\K$-algebra and $P\in\PS(A)$. The ideal $P$ is said to be \emph{locally closed} if the point $\{P\}$ is a locally closed point of $\PS(A)$. Let $B$ a be Poisson algebra. The Poisson centre of $B$ is the Poisson subalgebra $Z_P(B):=\{a\in B\ |\ \{a,-\}\equiv0\}$. The ideal $P$ is said to be \emph{Poisson rational} provided the field $Z_P\big(\F(A/P)\big)$ is algebraic over the ground field $\K$. For $J$ an ideal of $A$, there is a largest Poisson ideal contained in $J$ that is called the \emph{Poisson core} of $J$. Poisson cores of maximal ideals of $A$ are called \emph{Poisson primitive ideals}. We say that the Poisson Dixmier-Moeglin equivalence holds for the Poisson algebra $A$ if the following sets coincide:
\begin{enumerate}[topsep=0pt,itemsep=-1ex,partopsep=1ex,parsep=1ex]
	\item[(1)] the set of Poisson primitive ideals;
	\item[(2)] the set of locally closed Poisson ideals;
	\item[(3)] the set of Poisson rational ideals.
\end{enumerate}
It is shown in \cite{Oh2} that we have the inclusions $(2)\subseteq(1)\subseteq(3)$ for all affine Poisson algebras over a base field of characteristic zero. However the inclusion $(3)\subseteq(2)$ is not always satisfied as there exist counterexamples in all Krull dimension $d\geq4$ (see \cite{BLSM}). All algebras of the class $\mathcal{P}$ are affine Poisson algebras, therefore it only remains to show the inclusion $(3)\subseteq(2)$, as long as $\car\K=0$. It is known that Poisson affine spaces satisfy the Poisson Dixmier-Moeglin equivalence, see \cite[Example 4.6]{Goo1} for instance. In Section \ref{PDM} this fact together with the canonical embedding will allow us to prove the Poisson Dixmier-Moeglin equivalence for all algebras of the class $\mathcal{P}$. Even better, the Poisson primitive ideals are exactly the Poisson prime ideals that are maximal in their strata. Note that in \cite{GL} the Poisson Dixmier-Moeglin equivalence was shown for a class of Poisson algebras supporting rational torus actions. In our assumptions we do not require the existence of any torus action, and we indeed give an example (see Example \ref{examplePDME}), where previous results do not apply. 

	In \cite{Cau1} Cauchon uses his deleting derivations algorithm to obtain information on the spectra of the algebras of a class $\mathcal{R}$ of iterated Ore extensions (i.e. the algebras satisfying the hypotheses of \cite[Section 3.1]{Cau1}). These algebras are deformation of Poisson algebras of the class $\mathcal{P}$. More precisely, we are in the following setting. Let $R_t$ be an iterated Ore extension over $\K[t^{\pm1}]$:
\[R_t=\K[t^{\pm1}][x_1][x_2;\sigma_2,\Delta_2]\cdots[x_n;\sigma_n,\Delta_n],\]
such that for all $2\leq i\leq n$:
\begin{itemize}[noitemsep,topsep=0pt,parsep=2pt,partopsep=0pt]
\item $R_{t}^{<i}$ denotes the subalgebra of $R_t$ generated by $t^{\pm1},x_1,\dots,x_{i-1}$,
\item $\sigma_i$ is a $\K[t^{\pm1}]$-automorphism of $R_{t}^{<i}$ such that $\sigma_i(x_j)=t^{\la_{ij}}x_j$ for all $1\leq j<i$, where the scalars $\la_{ij}$ are integers,
\item $\Delta_i$ is a locally nilpotent $\K[t^{\pm1}]$-linear $\sigma_i$-derivation of $R_{t}^{<i}$,
\item $\sigma_i\Delta_i=t^{\eta_i}\Delta_i\sigma_i$ for some nonzero integer $\eta_i$,
\item $\Delta_i^k(R_{t}^{<i})\subseteq(t-1)^k(k)!_{t^{\eta_i}}R_{t}^{<i}$ for all $k\geq0$,
\item $A:=R_t/(t-1)R_t$ is commutative.
\end{itemize}

	 We fix a scalar $q\in\K^{\times}$ which is not a root of unity. Then, the algebra $R_q:=R_t/(t-q)R_t$ belongs to the class $\mathcal{R}$, and the algebra $A$ is a Poisson algebra which belongs to the class $\mathcal{P}$ (see \cite[Theorem 4.2]{LL}). The Poisson bracket on $A$ is given by the informal formula:
\begin{equation}
\label{scl}
\big\{r+(t-1)R_t,s+(t-1)R_t\big\}=\frac{rs-sr}{t-1}+(t-1)R_t,
\end{equation}
for all $r,s\in R_t$. We say that the algebra $R_q$ is a \textit{deformation} of the Poisson algebra $A$, and that $A$ is the \textit{semiclassical limit} of the algebra $R_t$ at $t-1$. The diagram of Figure \ref{diagsml} illustrates this situation.

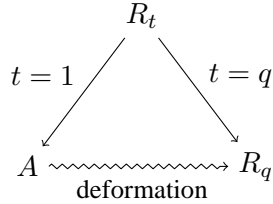
\begin{figure}[!ht]
	\centering
		\begin{tikzpicture}
	\node (a) at (1.5,0){$R_t$};
	\node (b) at (0,-2){$A$};
	\node (c) at (3,-2){$R_q$};
	\node (aa) at (0.2,-0.8){$t=1$};
	\node (bb) at (2.8,-0.8){$t=q$};
	\node (cc) at (1.5,-2.3){\small{deformation}};
	\draw [->] (a) to (b);
	\draw [->] (a) to (c);
	\draw [->,line join=round, decorate, decoration={	zigzag,
  									segment length=4,
    									amplitude=.9,post=lineto,
    									post length=2pt}]  (b) to (c);
		\end{tikzpicture}
	\caption{Deformation}
	\label{diagsml}
\end{figure}

In such a deformation-quantisation context, it is usually expected that the algebra $R_q$ and the Poisson algebra $A$ share similarities. For instance it is conjectured in \cite[Section 9.1]{Goo2} that there should be a homeomorphism between the spectrum of the generic quantised coordinate ring of an affine algebraic variety $V$ and the Poisson spectrum of its semiclassical limit $\mathcal{O}(V)$ when $\K$ is algebraically closed of characteristic zero. This conjecture has been investigated for several algebras, for instance we refer to the recent works \cite{Fr} and \cite{Yak}. In particular, building on previous work of Hodges-Levasseur and Joseph, progress have been made by Yakimov \cite{Yak} towards obtaining a homeomorphism between the symplectic leaves of a connected, simply connected complex algebraic group $G$ and the primitive spectrum of the quantized coordinate ring $R_q[G]$. 

	In light of this, it would be natural to ask whether or not there exists a homeomorphism between $\Sp(R_q)$ and $\PS(A)$. However, it is not always the case, and a counterexample is provided by the algebra $R_t$ generated over $\K[t^{\pm 1}]$ by $x$ and $y$ subject to:
\[xy-tyx=(t-1)^2.\]
In view of (\ref{scl}) we have:
\[\big\{x+(t-1)R_t,y+(t-1)R_t\big\}=(x+(t-1)R_t)(y+(t-1)R_t),\]
so that the algebra $A$ is a Poisson affine plane.
On the other hand the algebra $R_q$ is isomorphic to the first quantum Weyl algebra. In particular $R_q$ has a unique height one prime ideal, and its spectrum cannot be homeomorphic to the Poisson spectrum of $A$.

	In this article, we propose a sufficient condition for such a homeomorphism to exist (see Question \ref{question1} below). Toward describing this sufficient condition, we recall that Cauchon defines a partition of the spectrum $\Sp(R_q)$ indexed by a subset $W'$ of $W$, see \cite[Proposition 4.4.1]{Cau1}. Similarly, our algorithm allows us to define a partition of the Poisson spectrum $\PS(A)$ indexed by a subset $W'_P$ of $W$. Again it would be natural to ask whether or not these sets coincide, but the same example as above shows that it is not always the case. Indeed, from (\ref{psw}) it is clear that we have $W'_P=W=\big\{\emptyset,\{1\},\{2\},\{1,2\}\big\}$, whereas we have $W'=\big\{\emptyset,\{1\}\big\}$, by \cite[Section 7.2.1.2]{myphd} for instance. In Section \ref{PM}, we prove that $W'=W_P'$ for the algebra of $m \times p$ quantum matrices, and we use this fact to prove the following result.

\begin{thm}[Proposition \ref{www}] 
\label{thm-section5} Suppose that $\car\K=0$. Let $R=\mathcal{O}_q\big(\mathcal{M}_{m,p}(\K)\big)$ be the algebra of quantum matrices and $A=\mathcal{O}_q\big(\mathcal{M}_{m,p}(\K)\big)$ its semiclassical limit. There exists a bijection between $\Sp(R)$ and $\PS(A)$, which restricts to homeomorphisms between the strata $\Sp_w(R)$ and $\PS_w(A)$ for all $w\in W'=W_P'$.
\end{thm}

To prove this theorem we show that for $w\in W'=W_P'$ there is a homeomorphism between the strata $\Sp_w(R)$ and $\PS_w(A)$. However we deduce this homeomorphism from the canonical embedding, which is known to be continuous only when being restricted to a stratum. Therefore it is unclear whether the bijection of the theorem above is a homeomorphism or not. In small dimensions methods from \cite{BG2} and their Poisson analogues \cite{Fr} could be used to decide this question, but their computational nature would prevent use of them in the general case. 

In view of the above discussion, it is natural to ask the following question. 

\begin{ques}
\label{question1}
Let $R_t$ be an iterated Ore extension as above and suppose that $W'=W_P'$. Is there a homeomorphism between $\Sp(R_q)$ and $\PS(A)$?
\end{ques}

We note that the algebra generated over $\K[t^{\pm 1}]$ by $x$ and $y$ subject to $xy-tyx=(t-1)^2$ does not satisfy the condition that $W'=W_P'$.

\begin{rem}
\label{pp} The Poisson spectrum of a Poisson algebra is usually defined in a more general way. For a Poisson algebra $A$, a Poisson-prime ideal $P$ is a Poisson ideal such that if whenever $IJ\subseteq P$ for some Poisson ideals $I,J$ of $A$, then either $I\subseteq P$ or $J\subseteq P$. It is clear that a Poisson and prime ideal is a Poisson-prime ideal. If $A$ is noetherian and the characteristic of the base field is zero, then the converse is true thanks to \cite[Lemma 3.3.2]{Dix1}. 
The method we developed in Section \ref{ddap} does not apply to non prime Poisson-prime ideals. 	However our approach includes all the Poisson-prime ideals in the characteristic zero case and deals with a significant set of Poisson-prime ideals in positive characteristic. 

	In the situation described previously our main goal is to compare the spectrum of $R_q$ (for $q$ not a root of unity) with the Poisson spectrum of $A$. However even in the simplest example these spectra are not homeomorphic when we consider Poisson-prime ideals. Indeed, assume that $\K$ is algebraically closed and that $\car\K=p>0$. We denote by $R_t=\K[t^{\pm1}][x_1][x_2;\sigma_2]$ the iterated Ore extension such that $\sigma_2(t)=t$ and $\sigma_2(x_1)=tx_1$. Then $R_q$ is a quantum affine space for a non root of unity $q\in\K^{\times}$, and its set of prime ideals is well known, see \cite[II.1.2]{BG} for instance. In particular the principal ideals generated by $x_1$ and $x_2$ are the only height one prime ideals in $R_q$. The Poisson algebra $A$ is the Poisson affine space $A=\K[X_1,X_2]$ with $\{X_1,X_2\}=X_1X_2$. In addition of the ideals generated by $X_1$ and $X_2$, there are infinitely many other height one Poisson-prime ideals in $A$. For instance the ideal generated by the Poisson central element $X_1^p-1$ is a non prime Poisson ideal, and it follows from \cite[Lemma 3.5]{LL} that it is also Poisson-prime. Thus the set of Poisson-prime ideals of $A$ cannot be homeomorphic with the set of prime ideals in $R_q$. However it is easy to verify that there is a homeomorphism between the set of Poisson and prime ideals in $A$ and the set of prime ideals in $R_q$. 
	
	
	To summarise, when dealing with a Poisson algebra $B$ over a field of arbitrary characteristic we will restrict our attention to the study of the Poisson and prime ideals of $B$, and the set of such ideals will be denoted by $\PS(B)$.

\end{rem}

\section{Poisson deleting derivations algorithm}
\label{PDDA}

	The aim of this section is to define the Poisson deleting derivations algorithm. This algorithm is based on the Poisson deleting derivation homomorphism defined in \cite{LL}. We recall the definition and properties of this homomorphism in Section \ref{hom}, and introduce the class of Poisson algebras to which the Poisson deleting derivations algorithm applies in Section \ref{notation}.
	
\subsection{Poisson deleting derivation homomorphism}
\label{hom}

	Most of the definitions and results in this section are taken from \cite[Section 2]{LL}. We recall them here for the convenience of the reader. Poisson-Ore extensions are Poisson analogues of the well-known notion of Ore extension, or skew polynomial ring, in noncommutative ring theory. Their definition is based on the following result of Oh.
	
\begin{thm}
\label{Oh}
{\rm\cite[Theorem 1.1]{Oh1}}
Let $\al$ and $\de$ be $\mathbb{K}$-linear maps of a Poisson $\K$-algebra $A$. Then the polynomial algebra $R=A[X]$ is a Poisson algebra with Poisson bracket extending the Poisson bracket of $A$ and satisfying:
\begin{center}
$\{X,a\}=\al(a)X+\de(a)$ for all $a\in A$,
\end{center}
 if and only if $\al$ is a Poisson derivation of $A$, i.e. $\al$ is a $\mathbb{K}$-derivation of $A$ with:
\begin{center}
$\al(\{a,b\})=\{\al(a),b\}+\{a,\al(b)\}$ for all $a,b\in A$,
\end{center}
and 
$\de$ is a Poisson $\al$-derivation of $A$, i.e. $\de$ is a $\mathbb{K}$-derivation of $A$ with:
\begin{center}
$\de(\{a,b\})=\{\de(a),b\}+\{a,\de(b)\}+\al(a)\de(b)-\de(a)\al(b)$ for all $a,b\in A$.
\end{center}
\end{thm}

\begin{Def}
Let $A$ be a Poisson algebra. The set of Poisson derivations of $A$ is denoted by $\D_P(A)$. Let $\al\in\D_P(A)$ and $\de$ be a Poisson $\al$-derivation of $A$. Set $R=A[X]$. The algebra $R$ endowed with the Poisson bracket from Theorem \ref{Oh} is denoted by $R=A[X;\al,\de]_P$ and called a \textit{Poisson-Ore extension}. As usual we write $A[X;\al]_P$ for $A[X;\al,0]_P$.
\end{Def}

This construction is easily iterated. We say that \textit{$R$ is an iterated Poisson-Ore extension over $A$} if 
	\[R=A[X_1;\al_1,\de_1]_P[X_2;\al_2,\de_2]_P\cdots[X_n;\al_n,\de_n]_P\]
for some Poisson derivations $\al_1,\dots,\al_n$ and $\al_i$-Poisson derivations $\de_i$ ($1\leq i\leq n$) of the appropriate Poisson subalgebras.

	Let $\bo{\la}=(\lambda_{ij}) \in \mathcal{M}_n(\K)$ be a skew-symmetric matrix. Then we define a Poisson bracket on the polynomial algebra $\K[X_1,\dots,X_n]$ by setting by $\{X_i,X_j\}:=\la_{ij}X_iX_j$ for all $i,j$. This Poisson algebra is called the \textit{Poisson affine $n$-space associated to $\bo\la$} and is denoted by $\K_{\bo\la}[X_1,\dots,X_n]$. It is clear that the Poisson affine $n$-space $\K_{\bo\la}[X_1,\dots,X_n]$ is an iterated Poisson-Ore extension of the form: \[\K[X_1][X_2;\al_2]_P\cdots[X_n;\al_n]_P,\]
where $\al_i$ is the Poisson derivation of the Poisson algebra $\K[X_1][X_2;\al_2]_P\cdots[X_{i-1};\al_{i-1}]_P$ such that $\al_i(X_j)=\la_{ij}X_j$ for all $1\leq j<i\leq n$.

	The main tool to define the characteristic-free Poisson deleting derivations algorithm is the existence of higher derivations which are compatible with Poisson brackets. We now fix the notation and terminology used in this article.

\begin{Def}
\label{hd}
Let $A$ be a Poisson $\K$-algebra, $\al\in\D_P(A)$ and $\eta\in\K$.
\begin{enumerate}
	\item[(\rm{1})] A \textit{higher derivation} on $A$ is a sequence of $\K$-linear maps $(D_i)_{i=0}^{\infty}=(D_i)$ such that:
		\begin{center}
$D_0=\id_A$ and $D_n(ab)=\sum\limits_{i=0}^{n}D_i(a)D_{n-i}(b)$ for all $a,b\in A$ and all $n\geq0$.
		\end{center}
A higher derivation is \textit{iterative} if $D_iD_j=\binom{i+j}{i}D_{i+j}$ for all $i,j\geq0$, and \textit{locally nilpotent} if for all $a\in A$ there exists $n\geq0$ such that $D_i(a)=0$ for all $i\geq n$.
	\item[(\rm{2})] A higher derivation $(D_i)$ is a \textit{higher $\al$-skew Poisson derivation} if for all $a,b\in A$	and all $n\geq0$:
		\begin{center}
$D_n(\{a,b\})=\sum\limits_{i=0}^{n}\{D_i(a),D_{n-i}(b)\}+i\big(\al D_{n-i}(a)D_i(b)-D_i(a)\al D_{n-i}(b)\big)$.
		\end{center}
	\item[(\rm{3})] A higher $\al$-skew Poisson derivation is a \textit{higher $(\eta,\al)$-skew Poisson derivation} if for all $i\geq0$:
		\begin{center}
$D_i\al=\al D_i+i\eta D_i$.
		\end{center}
	\item[(\rm{4})] We say that the derivation $\de$ of a Poisson-Ore extension $A[X;\al,\de]_P$ \textit{extends to a higher $(\eta,\al)$-skew Poisson derivation} if there exists a higher $(\eta,\al)$-skew Poisson derivation $(D_i)$ on $A$ such that $D_1=\de$.
	\end{enumerate}
\end{Def}

We first observe that to define a higher derivation $(D_i)$ on an algebra $A$, it is enough to give its values on a set of generators of $A$. Moreover $(D_i)$ is iterative (resp. locally nilpotent) on $A$, if it is iterative (resp. locally nilpotent) on a set of generators of $A$. Tedious computations show that it is also enough to check assertions $(2)$ and $(3)$ of Definition \ref{hd} on a set of generators. 

\begin{ex}
Let $A=\K[X]$. We define a higher derivation $(D_i)$ on $A$ by setting:
\begin{align*}
D_i(X):=\left\{ \begin{array}{lll}
X & i=0,\\
1 & i=1,\\
0 & i>1.
\end{array} \right.
\end{align*}
Note that by induction we have:
\begin{align*}
\label{ik}
D_i(X^k)=\binom{k}{i} X^{k-i}
\end{align*}
for all $i,k\geq0$. It is clear that $D_i$ is iterative and locally nilpotent. Let $\al:=X\partial_X$ and $\de:=\partial_X$ where $\partial_X$ denotes the usual partial derivative of $A$ with respect to $X$. Then $\al$ is a Poisson derivation of $A$ and $\de$ is a Poisson $\al$-derivation of $A$ such that $D_1=\de$, where $A$ is endowed with the trivial Poisson structure ($\{a,b\}=0$ for all $a,b\in A$). It follows easily that $\de$ extends to an iterative, locally nilpotent higher $(1,\al)$-skew Poisson derivation on $A$.
\end{ex}

	We now recall the Poisson deleting derivation homomorphism that was defined in \cite[Section 2.3]{LL}. Note that a Poisson bracket extends uniquely by localisation \cite[Section 2.4.2]{LGPV}, so in particular the Poisson bracket of a Poisson-Ore extension $A[X;\al,\de]_P$ uniquely extends to the Laurent polynomial ring $A[X^{\pm1}]$, and we denote this Poisson algebra by $A[X^{\pm1};\al,\de]_P$.

\begin{thm}
\label{iso}
{\rm \cite[Theorem 2.11]{LL}}
Let $A[X;\al,\de]_P$ be a Poisson-Ore extension, where $A$ is a Poisson $\K$-algebra. Suppose that $\de$ extends to an iterative, locally nilpotent higher $(\eta,\al)$-skew Poisson derivation $(D_i)$ on $A$ such that $\eta\in\K^{\times}$. Then the algebra homomorphism $\theta: A\rightarrow A[X^{\pm1}]$ defined by:
	\[\theta(a)=\sum_{i\geq 0}\frac{1}{\eta^i}D_i(a)X^{-i}\]
uniquely extends to a Poisson $\K$-algebra isomorphism:
	\[\theta: A[Y^{\pm1};\al]_P\stackrel{\cong}{\longrightarrow} A[X^{\pm1};\al,\de]_P\]
by setting $\theta(Y)=X$.
\end{thm}

	We set $B:=A[X;\al,\de]_P$ and $S:=\{X^i\ |\ i\geq0\}$ so that we have $BS^{-1}=A[X^{\pm1};\al,\de]_P$. We deduce immediately the following result.

\begin{cor}
\label{equal}
$BS^{-1}$ contains a Poisson subalgebra $B'$ isomorphic to $A[Y;\al]_P$, and we have $B'S^{-1}=BS^{-1}$. In particular we have: 
\[\F\big( A[X;\al,\de]_P\big) =\F\big( B'\big) \cong \F\big( A[Y;\al]_P\big).\]
\end{cor}

\begin{proof}
Take $B':=\theta(A[Y;\al]_P)$.
\end{proof}

\subsection{A class of iterated Poisson-Ore extensions}
\label{notation}

In this section, we introduce the class of Poisson algebras that we will study in this paper. 

\begin{hypo}\ 
\label{hyp1}
\begin{itemize}
	\item[(\rm{1})] $A=\K[X_1][X_2;\al_2,\de_2]_P\cdots[X_n;\al_n,\de_n]_P$ is an iterated Poisson-Ore extension over $\K$. We set $A_i:=\K[X_1][X_2;\al_2,\de_2]_P\cdots[X_i;\al_i,\de_i]_P$ for all $1\leq i\leq n$.
	\item[(\rm{2})] Suppose that for all $1\leq j<i\leq n$ there exists $\la_{ij}\in\K$ such that $\al_i(X_j)=\la_{ij}X_j$. We set $\la_{ji}:=-\la_{ij}$ for all $1\leq j<i\leq n$.
	\item[(\rm{3})] For all $2\leq i\leq n$, assume that the derivation $\de_i$ extends to an iterative, locally nilpotent higher ($\eta_i,\al_i$)-skew Poisson derivation $(D_{i,k})_{k=0}^{\infty}$ on $A_{i-1}$, where $\eta_i$ is a \textbf{nonzero} scalar.
	\item[(\rm{4})] Assume that $\al_iD_{j,k}=D_{j,k}\al_i+k\la_{ij}D_{j,k}$ for all $2\leq j<i\leq n$ and all $k\geq0$.
\end{itemize}
\end{hypo}

\begin{notation}
We denote by $\mathcal{P}$ the class of iterated Poisson-Ore extensions which satisfy Hypothesis \ref{hyp1}.
\end{notation}

Note that, if $A=\K[X_1][X_2;\al_2,\de_2]_P\cdots[X_n;\al_n,\de_n]_P \in \mathcal{P}$, then the intermediate Poisson algebras $A_i$ from Hypothesis \ref{hyp1} also belong to $\mathcal{P}$. 

\begin{rem}
\label{caraczero}
In characteristic zero we have that $D_i=\frac{D_1^i}{i!}$ for all $i$ for any iterative higher derivation $(D_i)$. In particular it follows from \cite[Remark 5.1.2]{myphd} that in characteristic zero, one can replace assertions (3) and (4) of Hypothesis \ref{hyp1} by:
\begin{itemize}
\item[(\rm{3'})] Assume that for all $2\leq i\leq n$ the derivation $\de_i$ is locally nilpotent and that $\de_i\al-\al\de_i=\eta_i\de_i$ for some nonzero scalar $\eta_i$.
\end{itemize}
\end{rem}

	In the next sections we will need to use inductive arguments to define and study the Poisson deleting derivations algorithm. In the induction step we will need to re-arrange the order of the indeterminates of an iterated Poisson-Ore extension in $\mathcal{P}$. The following lemma will ensure that the new Poisson algebra is still in $\mathcal{P}$, so that one can apply the deleting derivation homomorphism to this new algebra, and thus proceed with the induction. In particular, to satisfy the hypothesis of Theorem \ref{iso} we need the scalars $\eta_i$ to be nonzero.

	The restriction of a linear map $f$ to a subspace $V$ of its domain will be denoted by $f|_V$.

\begin{lem}
\label{casn}
Let $A\in \mathcal{P}$ with $\de_{j+1}=\cdots=\de_{n}=0$. With the notation of Hypothesis \ref{hyp1}, we have the following.
\begin{enumerate}
	\item[\rm{(1)}] We can write $A=A_{j-1}[X_{j+1};\be_{j+1}]_P\cdots[X_n;\be_n]_P[X_j;\al'_j,\de'_j]_P$ where:
\begin{itemize}
		\item $\be_i|_{A_{j-1}}=\al_i|_{A_{j-1}}$ for all $j<i\leq n$ and $\be_i(X_l)=\la_{il}X_l$ for all $j<l<i$,
		\item $\al'_j|_{A_{j-1}}=\al_j$ and $\al'_j(X_l)=\la_{jl}X_l$ for all $j<l\leq n$,
		\item $\de'_j|_{A_{j-1}}=\de_j$ and $\de'_j(X_l)=0$ for all $j<l\leq n$.
\end{itemize}
	\item[\rm{(2)}] $\de'_j$ extends to an iterative, locally nilpotent higher ($\eta_j,\al'_j$)-skew Poisson derivation $(D'_{j,k})_{k=0}^{\infty}$ on $A_{j-1}[X_{j+1};\be_{j+1}]_P\cdots[X_n;\be_n]_P$ such that the restriction of $D'_{j,k}$ to $A_{j-1}$ coincides with $D_{j,k}$ for all $k\geq0$, and $D'_{j,k}(X_l)=0$ for all $k>0$ and all $j<l\leq n$.
\item[\rm{(3)}] $A=A_{j-1}[X_{j+1};\be_{j+1}]_P\cdots[X_n;\be_n]_P[X_j;\al'_j,\de'_j]_P$ also belongs to $\mathcal{P}$.
\end{enumerate}
\end{lem}

\begin{proof}
(1) Since $\{X_l,X_j\}=\la_{lj}X_lX_j$ for all $j<l\leq n$, the order of the variables $X_j,\dots,X_n$ can be changed. The resulting Poisson ($\al_i$-)derivations are those described above.\\
(2) This is an easy induction using \cite[Lemma 3.1]{LL}.\\
(3) This follows directly from (1) and (2).
\end{proof}

\subsection{Poisson deleting derivations algorithm}
\label{ppdda}

Let $A=\K[X_1][X_2;\al_2,\de_2]_P\cdots[X_n;\al_n,\de_n]_P \in \mathcal{P}$. We continue using the notation of Hypothesis \ref{hyp1}.

	We are now ready to describe the Poisson deleting derivations algorithm. For $j$ running from $n+1$ to $2$ we define, by a decreasing induction, a sequence $(X_{1,j},\dots,X_{n,j})$ of elements of $\F A$. First for $j=n+1$ we set $(X_{1,j},\dots,X_{n,j}):=(X_{1},\dots,X_{n})$. Then for $2\leq j\leq n$ we set:
	\begin{align*}
	X_{i,j}:=\left\{
\begin{array}{ll}
X_{i,j+1} &\quad i\geq j,\\
\sum\limits_{k\geq0}\frac{1}{\eta_j^k}D_{j,k}(X_{i,j+1})X_{j,j+1}^{-k} &\quad  i<j,
\end{array}
\right.
\end{align*}
for all $1\leq i\leq n$. Moreover for all $2\leq j\leq n+1$ we set $C_j:=\K[X_{1,j},\dots,X_{n,j}]$. In particular we have $C_{n+1}=A$. The following proposition describes explicitly the Poisson structures on the algebras $C_{j}$ induced by these changes of variables.

\begin{prop}
For all $1\leq j\leq n$ we have:
\begin{enumerate}
	\item[\rm{(1)}] \textit{$C_{j+1}$ is isomorphic to an iterated Poisson-Ore extension of the form:	
\[\K[X_1]\cdots[X_j;\al_j,\de_j]_P[X_{j+1};\be_{j+1}]_P\cdots[X_n;\be_n]_P\]
by a Poisson isomorphism sending $X_{i,j+1}$ to $X_i$ for $1\leq i\leq n$.}
	\item[\rm{(2)}] \textit{For all $l\in\{j+1,\dots,n\}$, the map $\be_l$ is a Poisson derivation such that $\be_l(X_i)=\la_{li}X_i$ for all $1\leq i<l$ and we have $\be_lD_{i,k}=D_{i,k}\be_l+k\la_{li}D_{i,k}$ for all $1<i\leq j$ and all $k\geq0$.}
	\item[\rm{(3)}] Set $S_j=\{U_j^m\ |\ m\geq0\}=\{V_j^m\ |\ m\geq0\}$. We have $C_jS_j^{-1}=C_{j+1}S_j^{-1}$.
\end{enumerate}
\end{prop}
	
\begin{proof}
We proceed by a decreasing induction on $j$. For $j=n+1$ we have $C_{n+1}=A$ and the result follows from Hypothesis \ref{hyp1}. We now suppose that the result is true for a rank $j+1>2$. To simplify notation we set $U_i=X_{i,j+1}$ and $V_i=X_{i,j}$ for all $1\leq i\leq n$. By the induction hypothesis we can express $C_{j+1}$ as the iterated Poisson-Ore extension:
		\[\K[U_1]\cdots[U_j;\al_j,\de_j]_P[U_{j+1};\be_{j+1}]_P\cdots[U_n;\be_n]_P\in\mathcal{P}.\]
By Lemma \ref{casn} we can write:
	\[C_{j+1}=\K[U_1]\cdots[U_{j-1};\al_{j-1},\de_{j-1}]_P[U_{j+1};\be'_{j+1}]_P\cdots[U_n;\be'_n]_P[U_j;\al'_j,\de'_j]_P,\]
where $\be'_l$ for all $j<l\leq n$ and $\al'_j$ and $\de'_j$ are defined as in assertion (1) of Lemma \ref{casn}. In particular $\de'_j$ extends to an iterative, locally nilpotent higher $(\eta_j,\al'_j)$-skew Poisson derivation $(D'_{j,k})_{k=0}^{\infty}$ on the Poisson algebra:
\[\widehat{C_{j+1}}:=\K[U_1]\cdots[U_{j-1};\al_{j-1},\de_{j-1}]_P[U_{j+1};\be'_{j+1}]_P\cdots[U_n;\be'_n]_P.\]
Therefore by applying Theorem \ref{iso} to the Poisson algebra $C_{j+1}=\widehat{C_{j+1}}[U_j;\al'_j,\de'_j]_P$ we get a Poisson algebra isomorphism $\theta$ from $\widehat{C_{j+1}}[U_j^{\pm1};\al'_j]_P$ to $\widehat{C_{j+1}}[U_j^{\pm1};\al'_j,\de'_j]_P$ sending $U_j$ to $U_j$. In particular we have $\theta(U_i)=V_i$ for all $1\leq i\leq n$ since $U_j=V_j$ and:
\begin{align*}
	\theta(U_i)=\sum_{l\geq0}\frac{1}{\eta_j^l}D'_{j,l}(U_i)U_j^{-l}
	=\left\{
\begin{array}{ll}
\sum\limits_{l\geq0}\frac{1}{\eta_j^l}D_{j,l}(U_i)U_j^{-l} &\quad i<j,\\
U_i &\quad  i<j.
\end{array}
\right.
\end{align*}
Thus we have:
\[\theta\big(\widehat{C_{j+1}}[U_j;\al'_j]_P\big)=\K[V_1]\cdots[V_{j-1};\al_{j-1},\de_{j-1}]_P[V_{j+1};\be'_{j+1}]_P\cdots[V_n;\be'_n]_P[V_j;\al'_j]_P=C_j,\]
and by Corollary \ref{equal} we get $C_jS_j^{-1}=C_{j+1}S_j^{-1}$. This proves assertion (3). 

	Since $\{V_l,V_j\}=\la_{lj}V_jV_l$ for all $j<l\leq n$ we can bring back $V_j$ in the $j$-th position:	\[C_j=\K[V_1]\cdots[V_{j-1};\al_{j-1},\de_{j-1}]_P[V_j;\be''_j]_P\cdots[V_n;\be''_n]_P,\]
where for all $j\leq l\leq n$, the map $\be''_l$ is a Poisson derivation such that $\be''_l(V_i)=\la_{li}V_i$ for all $1\leq i<l$. This proves assertion (1).

	Finally, the fact that $\be''_lD_{m,k}=D_{m,k}\be''_l+k\la_{lm}D_{m,k}$ for all $1<m<j\leq l\leq n$ and all $k\geq0$, follows directly from the equalities:
\begin{itemize}
\item $\be_lD_{m,k}=D_{m,k}\be_l+k\la_{lm}D_{m,k}$ for all $1<m\leq j<l\leq n$ and all $k\geq0$,
\item $\be_l(U_i)=\la_{li}U_i$ for all $j<l\leq n$ and all $1\leq i<l$,
\item $\al_j(U_i)=\la_{ji}U_i$ for all $1\leq i<j$,
\item $\be''_l(V_i)=\la_{li}V_i$ for all $j\leq l\leq n$ and all $1\leq i<l$.
\end{itemize}
This proves assertion (2).
\end{proof}

\begin{cor}
The algebra $\ov{A}:=C_2$ is a Poisson affine space. More precisely, by setting $T_i:=X_{i,2}$ for all $1\leq i\leq n$ and $\bo\la$ for the skew-symmetric matrix defined by $\bo\la:=(\la_{ij})\in\mathcal{M}_n(\K)$ we have:
\[\ov{A}=\K_{\bo\la}[T_1,\dots,T_n].\] 
\end{cor}

\section{Poisson deleting derivations algorithm and Poisson spectrum}
\label{ddap}

	Recall that for a Poisson algebra $B$ we denote by $\PS(B)$ its Poisson spectrum, i.e. the set of prime ideals of $B$ which are also Poisson ideals. $\PS(B)$ is endowed with the induced Zariski topology. In this section we focus on the behaviour of the Poisson spectrum of an iterated Poisson-Ore extension $A \in \mathcal{P}$ under the Poisson deleting derivation algorithm. We show that there is an embedding between $\PS(A)$ and $\PS(\ov{A})$. This is done by showing that, at each step of the algorithm there is an embedding between $\PS(C_{j+1})$ and $\PS(C_j)$ for all $2\leq j \leq n$.
	
	Throughout this section, we use the notation of Hypothesis \ref{hyp1} and we fix $2\leq j\leq n$, and set $U_i:=X_{i,j+1}$ and $V_i:=X_{i,j}$ for all $1\leq i\leq n$.

\subsection{The embedding $\varphi_j:\PS(C_{j+1})\rightarrow\PS(C_j)$}
\label{gg}

	Recall that $U_j=V_j$, and set:
\begin{alignat*}{2}
	& \mathcal{P}_j^0(C_j)=\{P\in\PS(C_j)\ |\ V_j\notin P\}, \quad\quad\quad && \mathcal{P}_j^1(C_j)=\{P\in\PS(C_j)\ |\ V_j\in P\},\\
	& \mathcal{P}_j^0(C_{j+1})=\{P\in\PS(C_{j+1})\ |\ U_j\notin P\},  && \mathcal{P}_j^1(C_{j+1})=\{P\in\PS(C_{j+1})\ |\ U_j\in P\}.
\end{alignat*}
These sets partition $\PS(C_j)$ and $\PS(C_{j+1})$. Since we have $C_jS_j^{-1}=C_{j+1}S_j^{-1}$, contraction and extension of ideals provide bijections between $\mathcal{P}_j^0(C_{j})$ and $\mathcal{P}_j^0(C_{j+1})$ (it is easy to show that the contraction or the extension of a Poisson ideal is again a Poisson ideal). More precisely we have the following result.

\begin{lem}
\label{all}
There is a homeomorphism $\varphi_j^0:\mathcal{P}_j^0(C_{j+1})\rightarrow\mathcal{P}_j^0(C_j)$ given by $\varphi_j^0(P):=PS_j^{-1}\cap C_j$ for $P\in\mathcal{P}_j^0(C_{j+1})$. Its inverse is defined by $(\varphi_j^0)^{-1}(Q):=Q S_j^{-1}\cap C_{j+1}$ for $Q\in\mathcal{P}_j^0(C_{j})$.
\end{lem}

We note that both $\varphi_j^0$ and $(\varphi_j^0)^{-1}$ respect the inclusion of Poisson prime ideals. We now want to compare $\mathcal{P}_j^1(C_{j+1})$ and $\mathcal{P}_j^1(C_{j})$. For, we denote by $\langle U_j \rangle_P$ the smallest Poisson ideal in $C_{j+1}$ containing $U_j$ and for all $1\leq i\leq n$, we denote by $\ov{U_i}$ the image of $U_i$ in the Poisson algebra $C_{j+1}/\langle U_j\rangle_P$.

\begin{lem}
\label{g}
There is a surjective Poisson algebra homomorphism $g_j:C_j\rightarrow C_{j+1}/\langle U_j\rangle_P$ given by $g_j(V_i)=\ov{U_i}$ for all $1\leq i\leq n$.
\end{lem}

\begin{proof}
The map $g_j$ is the composition of the canonical quotient map $\pi:C_{j+1}\rightarrow C_{j+1}/\langle U_j\rangle_P$ and the algebra isomorphism $\Psi:C_j\rightarrow C_{j+1}$ defined by $\Psi(V_i)=U_i$ for all $1\leq i\leq n$. Thus clearly $g_j=\pi\circ\Psi$ is a surjective algebra homomorphism. Note that $\pi$ is a Poisson algebra homomorphism whereas $\Psi$ is not in general, so we cannot conclude directly. We show that $g_j(\{V_k,V_l\})=\{g_j(V_k),g_j(V_l)\}$ for all $1\leq l<k\leq n$. First if $k\geq j$ we have:
	\[g_j(\{V_k,V_l\})=g_j(\la_{kl}V_kV_l)=\la_{kl}\ov{U_k}\ov{U_l}=\{\ov{U_k},\ov{U_l}\}=\{g_j(V_k),g_j(V_j)\}.\]
(Note that when $k=j$ we have $\ov{U_k}=0$). If $k<j$ we have $\Psi(\de_k(V_l))=\de_k(U_l)$ and thus:	
\begin{align*}	
				g_j(\{V_k,V_l\})=g_j\big(\la_{kl}V_kV_l+\de_{k}(V_l)\big)&=\la_{kl}\ov{U_k}\ov{U_l}+g_j\big(\de_{k}(V_l)\big)\\
				&=\la_{kl}\ov{U_k}\ov{U_l}+\ov{\de_{k}(U_l)}=\{\ov{U_k},\ov{U_l}\}=\{g_j(V_k),g_j(V_l)\}.
\end{align*}
\end{proof}

Set $N_j:=\ker(g_j)$. There is a homeomorphism $\varphi_j^1$ from $\mathcal{P}_j^1(C_{j+1})$ to $\{P\in\PS(C_j)\ |\ N_j\subseteq P\}$ defined by $\varphi_j^1(P):=g_j^{-1}(P/\langle U_j\rangle_P)$ for $P\in\mathcal{P}_j^1(C_{j+1})$. Since $V_j=U_j\in N_j$ we have $\{P\in\PS(C_j)\ |\ N_j\subseteq P\}\subseteq\mathcal{P}_j^1(C_{j})$ and:

\begin{lem}
\label{bee}
There is an increasing and injective map $\varphi_j^1:\mathcal{P}_j^1(C_{j+1})\rightarrow\mathcal{P}_j^1(C_{j})$ defined by $\varphi_j^1(P)=g_j^{-1}(P/\langle U_j\rangle_P)$ for $P\in\mathcal{P}_j^1(C_{j+1})$, which induces a homeomorphism on its image.
\end{lem}

We can now define a map $\varphi_j:\PS(C_{j+1})\rightarrow\PS(C_j)$ by setting:
	\[\varphi_j(P)=\left\{
		\begin{array}{ll}
		\varphi_j^0(P)\mbox{ if }P\in \mathcal{P}_j^0(C_{j+1}),\\
		\varphi_j^1(P)\mbox{ if }P\in \mathcal{P}_j^1(C_{j+1}).
		\end{array}
		\right.
\]

As a direct consequence of Lemmas \ref{all} and \ref{bee} we get the following result.

\begin{prop}
\label{inj}
The map $\varphi_j:\PS(C_{j+1})\rightarrow\PS(C_j)$ is injective. For $\varepsilon\in\{0,1\}$, the map $\varphi_j$ induces a homeomorphism from $\mathcal{P}_j^\varepsilon(C_{j+1})$ to $\varphi_j\big(\mathcal{P}_j^\varepsilon(C_{j+1})\big)$ which is a closed subset of $\mathcal{P}_j^\varepsilon(C_{j})$.
\end{prop}


\subsection{The canonical partition of $\PS(A)$}
\label{injec}

\begin{Def}
We set $\varphi:=\varphi_2\circ\cdots\circ\varphi_n$. This is an injective map from $\PS(C_{n+1})=\PS(A)$ to $\PS(C_{2})=\PS(\ov{A})$ and we refer to it as the \emph{canonical embedding}.
\end{Def}

Let $W:=\mathscr{P}([\mspace{-2 mu} [1,n] \mspace{-2 mu} ])$ denote the powerset of $[\mspace{-2 mu} [1,n] \mspace{-2 mu} ]$. For $w\in W$, we set:
	\[\PS_w(\ov{A}):=\big\{Q\in\PS(\ov{A})\ |\ Q\cap\{T_1,\dots,T_n\}=\{T_i\ |\ i\in w\}\big\},\]
where we recall that the $T_i$ are the generators of the Poisson affine space $\ov{A}$. Note that these sets form a partition of $\PS(\ov{A})$. For all $w\in W$ we set:
$$\PS_w(A):=\varphi^{-1}\big(\PS_w(\ov{A})\big),$$
 and $W'_P$ for the set of $w$ such that $\PS_w(A)\neq\emptyset$, i.e.
 $$W'_P:=\{ w \in W ~|~ \PS_w(A)\neq\emptyset \}.$$
  This family forms a partition of $\PS(A)$:
	\[\PS(A)=\displaystyle\bigsqcup\limits_{w\in W'_P}\PS_w(A)\quad \mbox{ and } \quad  |W'_P|\leq|W|=2^n.\]

\begin{Def}
This partition of $\PS(A)$ will be called the {\em canonical partition}, the elements of $W'_P$ will be called the {\em Cauchon diagrams associated to A}, or Cauchon diagrams for short. Finally, for $w\in W'_P$ the set $\PS_w(A)$ is called the {\em stratum} associated to $w$.
\end{Def}

Note that the set $W'_P$ depends on the expression of $A$ as an iterated Poisson-Ore extension.

\subsection{A membership criterion for $\Imm(\varphi)$}
\label{critim}

	The following results help us to understand whether or not a given Poisson prime ideal of $\ov{A}$ belongs to the image of the canonical embedding. This will be useful to understand better the canonical partition and when dealing with examples. We start this section with a membership criterion for $\Imm(\varphi_j)$. Recall that $N_j=\ker(g_j)$ was defined in Section \ref{gg}.
\begin{lem}
\label{eqq}
Let $Q\in\PS(C_j)$. Then:
\begin{align*}
		Q\in\Imm(\varphi_j) &\Leftrightarrow \big(\text{ either }U_j=V_j\notin Q \text{, or } N_j\subseteq Q\big).
\end{align*}
\end{lem}

\begin{proof}
This is clear since the map $\varphi^0_j$ is a bijection from $\mathcal{P}_j^0(C_{j+1})$ to $\mathcal{P}_j^0(C_{j})$ and the map $\varphi^1_j$ is a bijection from $\mathcal{P}_j^1(C_{j+1})$ to $\{Q\in\PS(C_j)\ |\ N_j\subseteq Q\}$.
\end{proof}

Set $f_1:=\id_{\text{\tPS}(\ov{A})}$. For all $2\leq j\leq n$ we define a map $f_j:\PS(C_{j+1})\rightarrow\PS(\ov{A})$ by setting $f_j:=f_{j-1}\circ\varphi_j$. Note that each $f_j$ is injective. We deduce from Lemma \ref{eqq} the following membership criterion for $\Imm(\varphi)$.

\begin{prop}
Let $Q\in\PS(\ov{A})$. The following are equivalent:
\begin{itemize}
	\item $Q\in\Imm(\varphi)$,
	\item for all $2\leq j\leq n$ we have $Q\in\Imm(f_{j-1})$ and\\
either $X_{j,j}=X_{j,j+1}\notin f_{j-1}^{-1}(Q)$, or $N_{j}\subseteq f_{j-1}^{-1}(Q)$.
\end{itemize}
\end{prop}

\begin{rem}
\label{ppss}
To understand $N_j$ it is enough to understand $\langle U_j\rangle_P$ since $N_j=\Psi^{-1}(\langle U_j\rangle_P)$, where the algebra isomorphism $\Psi:C_j\rightarrow C_{j+1}$ is defined by $\Psi(V_i)=U_i$ for all $1\leq i\leq n$ (see proof of Lemma \ref{g}). As $\{U_j,U_i\}=\la_{ji}U_jU_i+\de_j(U_i)$ for all $i\in[\mspace{-2 mu} [1,j-1] \mspace{-2 mu} ]$, we deduce that:
	\[\langle U_j,\de_j(U_i)\ |\ i\in[\mspace{-2 mu} [1,j-1] \mspace{-2 mu} ]\rangle\subseteq\langle U_j\rangle_P.\]
By minimality of $\langle U_j\rangle_P$, the reverse inclusion will be satisfied if the left hand side is a Poisson ideal. However this is not always the case as the following example demonstrates. Let $A$ be the iterated Poisson-Ore extension $A:=\C[X][Y;\be,\Delta]_P[Z;\al,\de]_P$, where $\be:=-X\partial_X$, $\al:=X\partial_X-Y\partial_Y$, $\Delta:=\partial_X$ and $\de:=Y^2\partial_X$, so that:
\begin{align*}
&\{Y,X\}=-XY+1,\\
&\{Z,X\}=XZ+Y^2,\\
&\{Z,Y\}=-YZ.
\end{align*}
We have $\Delta\be-\be\Delta=-\Delta$ and $\de\al-\al\de=\de$. Moreover since $\Delta$ and $\de$ are locally nilpotent, assertion (3') is satisfied and the algebra $A$ belongs to $\mathcal{P}$. However the ideal $\langle Z,Y^2\rangle$ is not a Poisson ideal.
\end{rem}

\subsection{Topological and algebraic properties of the canonical embedding}
\label{taprop}

	In this section we investigate topological properties of the canonical embedding. We start with some results that will be used in this section as well as latter on.

\begin{lem}
\label{iff}
Let $l\in\{j\dots,n\}$, $P\in\PS(C_{j+1})$ and $Q:=\varphi_j(P)\in\PS(C_j)$. Then we have:
	\[U_l\in P\quad\Leftrightarrow\quad V_l\in Q.\]
\end{lem}

\begin{proof}
	If $l=j$, then $(U_l\in P)\Leftrightarrow\big(P\in\mathcal{P}_j^1(C_{j+1})\big)$ and $(V_l\in Q)\Leftrightarrow\big(Q\in\mathcal{P}_j^1(C_{j})\big)$, and the result is given by Proposition \ref{inj}. We distinguish between two cases when $l>j$. First, if $P\in\mathcal{P}_j^0(C_{j+1})$, then we have:
	\[U_l\in P\quad \Rightarrow\quad U_l\in PS_j^{-1}\quad \Rightarrow\quad V_l=U_l\in C_j\cap PS_j^{-1}=Q,\]
and
	\[V_l\in Q\quad \Rightarrow\quad V_l\in QS_j^{-1}\quad \Rightarrow\quad U_l=V_l\in C_{j+1}\cap QS_j^{-1}=P.\]
Next, if $P\in\mathcal{P}_j^1(C_{j+1})$, then we have:
	\[U_l\in P\quad \Leftrightarrow\quad \ov{U_l}\in\frac{P}{\langle U_j\rangle_P}\quad \Leftrightarrow\quad g_j(V_l)\in\frac{P}{\langle U_j\rangle_P}\quad \Leftrightarrow\quad V_l\in g_j^{-1}\Big(\frac{P}{\langle U_j\rangle_P}\Big)=Q.\]
\end{proof}

 	For $Q\in\Imm(\varphi)$, we set $P_j:=f_{j-1}^{-1}(Q)\in\PS(C_j)$ for all $2\leq j\leq n+1$. In particular, note that $Q=P_2$.

\begin{cor}
\label{belongs}
Let $i\in\{1,\dots,n\}$ and $Q\in\Imm(\varphi)$. We have:
	\[T_i=X_{i,2}\in P_2 \Leftrightarrow X_{i,i+1}\in P_{i+1}.\]
\end{cor}

\begin{proof}
This follows by induction from Lemma \ref{iff}.
\end{proof}

 Let $1\leq j\leq n$ and $w\in W$. Set $X_w:=f_j^{-1}(\PS_w(\ov{A}))\subset\PS(C_{j+1})$. When $j\geq2$, we also set $Y_w:=f_{j-1}^{-1}(\PS_w(\ov{A}))\subset\PS(C_{j})$, so that $X_w=\varphi_j^{-1}(Y_w)$ since $f_j=f_{j-1}\circ\varphi$. Note that the sets $X_w$ and $Y_w$ can be empty.
	
\begin{lem}
\label{w}
For $j\leq l\leq n$ we have:
\begin{itemize}
	\item If $l\notin w$, then $U_l\notin P$ for all $P\in X_w$,
	\item If $l\in w$, then $U_l\in P$ for all $P\in X_w$.
\end{itemize}
\end{lem}

\begin{proof}
Note that since $l\geq j$ we have $U_l=X_{l,k}=T_l$ for all $2\leq k\leq j+1$. If $j=1$, we have $X_w=\PS_w(\ov{A})$ and the result comes from the definition of $\PS_w(\ov{A})$.

	Assume that $j\geq2$ and the result shown for $j-1$. First assume that $l\notin w$ and let $P\in X_w$. If $U_l\in P$ then $V_l\in Q=\varphi_j(P)\in Y_w$ by Lemma \ref{iff}. This contradicts the induction hypothesis, thus $U_l\notin P$. Next assume that $l\in w$ and let $P\in X_w$. If $U_l\notin P$ then $V_l\notin Q=\varphi_j(P)\in Y_w$ by Lemma \ref{iff}. This contradicts the induction hypothesis, thus $U_l\in P$.
\end{proof}

\begin{lem}
\label{lem8}
The set $f_j(X_w)$ is a closed subset of $\PS_w(\ov{A})$, and $f_j$ induces (by restriction) a homeomorphism from $X_w$ to $f_j(X_w)$.
\end{lem}

\begin{proof}
The result is trivial if $j=1$. Assume that $j\geq 2$ and that the result is shown for $j-1$. By Lemma \ref{w} (applied to $l=j$ for $j$ and $j-1$) we have:
\begin{itemize}
	\item $(j\notin w)\quad\Rightarrow\quad (X_w\subset\mathcal{P}_j^0(C_{j+1})\mbox{ and }Y_w\subset\mathcal{P}_j^0(C_{j}))$,
	\item $(j\in w)\quad\Rightarrow\quad (X_w\subset\mathcal{P}_j^1(C_{j+1})\mbox{ and }Y_w\subset\mathcal{P}_j^1(C_{j}))$. 
\end{itemize}
	Therefore we have $\varphi_j(X_w)=Y_w\cap Z$ where $Z=\varphi_j(\mathcal{P}_j^\varepsilon(C_{j+1}))$ with $\varepsilon\in\{0,1\}$. By Proposition \ref{inj}, $Y_w\cap Z$ is a closed subset of $Y_w$, and $\varphi_j$ induces a homeomorphism from $X_w$ to $Y_w\cap Z$. By the induction hypothesis $f_{j-1}$ induces a homeomorphism from $Y_w$ to $f_{j-1}(Y_w)$ which is a closed subset of $\PS_w(\ov{A})$.
	
	Thus $f_{j-1}(Y_w\cap Z)$ is a closed subset of $f_{j-1}(Y_w)$ (as the image of a closed subset by a homeomorphism), and so is a closed subset of $\PS_w(\ov{A})$. Since $f_j(X_w)=f_{j-1}\circ\varphi_j(X_w)=f_{j-1}(Y_w\cap Z)$, the first assertion is proved.
	
	The map $f_j:X_w\rightarrow f_j(X_w)=f_{j-1}(Y_w\cap Z)$ is the composition of the two maps $\varphi_j:X_w\rightarrow Y_w\cap Z$ and $f_{j-1}:Y_w\cap Z\rightarrow f_{j-1}(Y_w\cap Z)$ which are both homeomorphisms.
\end{proof}

	When $j=n$ we have $f_j=\varphi$ and $X_w=\PS_w(A)$, for all $w\in W$. We deduce the following result.

\begin{thm}
\label{topo}
Let $\varphi:\PS(A)\rightarrow\PS(\ov{A})$ be the canonical embedding and $w\in W'_P$. Then $\varphi(\PS_w(A))$ is a (non empty) closed subset of $\PS_w(\ov{A})$, and $\varphi$ induces (by restriction) a homeomorphism from $\PS_w(A)$ to $\varphi(\PS_w(A))$.
\end{thm}

In particular we note that the map $\varphi$ respect the inclusion of Poisson prime ideals within the same strata. In a lot of examples (when the Poisson algebra considered is supporting a suitable torus action for instance) the inclusion of the previous theorem is actually an equality:  
\[\varphi(\PS_w(A))=\PS_w(\ov{A}).\] 
However this is not true in general as the following example demonstrates. 

\begin{ex}
\label{ex4}
Assume that $\car\K=0$. Let $B=\K_{\bo\la}[X_1,X_2,X_3]$ be the Poisson affine space where:
\[\bo\la=\begin{pmatrix}
0&0&-1\\
0&0&-1\\
1&1&0\\
\end{pmatrix}.
\]
Observe that $\al:=-X_1\frac{\partial}{\partial X_1}-X_2\frac{\partial}{\partial X_2}$ is a Poisson derivation of $B$ and $\de:=(X_1+X_2)\frac{\partial}{\partial X_3}$ a Poisson $\al$-derivation of $B$. Thus we can form the Poisson-Ore extension $A=B[X_4;\al,\de]_P$. Note that $\de$ is locally nilpotent and that we have $\de\al=\al\de+\de$. Thus $A\in\mathcal{P}$ by Remark \ref{caraczero}. In particular the derivation $\de$ uniquely extends to an iterative, locally nilpotent higher $(1,\al)$-skew Poisson derivation $(D_i)$ defined by $D_i=\frac{\de^i}{i!}$ for all $i\geq0$. Therefore we can apply the deleting derivations algorithm (actually the deleting derivation homomorphism is enough here since there is only one step in the algorithm). 

	The Poisson algebra $\ov{A}$ is the Poisson affine space $\K_{\bo{\la'}}[T_1,T_2,T_3,T_4]$ where:
\[\bo{\la'}=\begin{pmatrix}
0&0&-1&1\\
0&0&-1&1\\
1&1&0&0\\
-1&-1&0&0\\
\end{pmatrix},
\]
and where $T_1=X_1$, $T_2=X_2$, $T_3=X_3+(X_1+X_2)X_4^{-1}$ and $T_4=X_4$. The canonical embedding is the map $\varphi$ from $\PS(A)$ to $\PS(\ov{A})$ defined by:
\begin{align*}
	P\longmapsto  \left\{
	\begin{array}{ll}
		PS^{-1}\cap\ov{A}\qquad& X_4\notin P\\
		g^{-1}(P/\langle X_4\rangle_P) & X_4\in P,
	\end{array}
\right.
\end{align*}
where $S$ is the multiplicative set of $A$ generated by $X_4$, and where:
\begin{align*}
	g\ :\ \ \ \ov{A} \ \ \  &\longrightarrow \frac{A}{\langle X_4\rangle_P}\\
					T_i	\ \ \	 &\longmapsto X_i+\langle X_4\rangle_P \quad \text{for}\quad i=1,\dots,4.
\end{align*}

	Firstly we show that $\{4\}\in W_P'\subseteq W=\mathscr{P}([\mspace{-2 mu} [1,4] \mspace{-2 mu} ])$. Set $P:=\langle X_4\rangle_P=\langle X_4,X_1+X_2\rangle$. It easy to see that $P\in\PS(A)$. Since $X_4\in P$, Lemma \ref{d=0} gives us a Poisson algebra isomorphism $A/P\cong\ov{A}/\varphi(P)$ sending $X_i+P$ to $T_i+\varphi(P)$ for $1\leq i\leq 4$. Therefore we have $T_4\in\varphi(P)$ and $T_1,T_2,T_3\notin\varphi(P)$. Hence $\varphi(P)\in\PS_{\{4\}}(\ov{A})$ and $\{4\}\in W'_P$. 

	Secondly, since $\{4\}\in W'_P$, Theorem \ref{topo} tells us that the set $\varphi\big(\PS_{\{4\}}(A)\big)$ is a non-empty closed subset of $\PS_{\{4\}}(\ov{A})$. We will show that this inclusion is strict. For $Q\in\PS_{\{4\}}(A)$ we have $T_4\in\varphi(Q)\in\PS_{\{4\}}(\ov{A})$, so $X_4\in Q$. But then $Q\in\mathcal{P}^1(A)$ and thus $\langle T_4,T_1+T_2\rangle\subseteq \varphi(Q)$. Hence we have the following inclusion:
	\[\varphi\big(\PS_{\{4\}}(A)\big)\subseteq\{P\in\PS_{\{4\}}(\ov{A})\ |\ T_4\in P,\ T_1+T_2\in P\}\subseteq \PS_{\{4\}}(\ov{A}).\]

But it is clear that $\langle T_4\rangle\in\PS_{\{4\}}(\ov{A})$. Thus:
	\[\varphi\big(\PS_{\{4\}}(A)\big)\varsubsetneq\PS_{\{4\}}(\ov{A}).\]
\end{ex}

To conclude this section we prove the following criterion for a Poisson prime ideal to belong to the image of the canonical embedding.

\begin{prop}
\label{wPQ}
Let $w\in W_P'$, $P\in\PS_w(A)$ and $Q\in\PS_w(\ov{A})$ such that $\varphi(P)\subseteq Q$. Then $Q\in\Imm(\varphi)$.
\end{prop}
\begin{proof}
We prove by induction that $Q\in\Imm(f_j)$ for all $1\leq j \leq n$. When $j=1$ the result is trivial since $f_1$ is the identity on $\PS(\ov{A})$. Suppose that $Q\in\Imm(f_{j-1})$ for some $2\leq j\leq n$. We have to show that $f_{j-1}^{-1}(Q)\in\Imm(\varphi_j)$ since $f_j=f_{j-1}\circ\varphi_j$. Firstly we remark that $\varphi(P)\subseteq Q$ implies that $f_{j-1}^{-1}(\varphi(P))\subseteq f_{j-1}^{-1}(Q)$ by Lemma \ref{lem8} (with $j$ replaced by $j-1$). We now distinguish between two cases. 

	Assume that $U_j\notin f_{j}^{-1}(\varphi(P))$. Then by Corollary \ref{belongs} we have $T_j\notin\varphi(P)$ and 
	so $j\notin w$. But then by Lemma \ref{w} we have $U_j \notin f_{j-1}^{-1}(Q)$ and thus $f_{j-1}^{-1}(Q)\in\Imm(\varphi_j)$ by Lemma \ref{eqq}.
	
	Assume that $U_j\in f_{j}^{-1}(\varphi(P))$. Then: \[N_j\subseteq \varphi_j\big(f_{j}^{-1}(\varphi(P))\big)=f_{j-1}^{-1}(\varphi(P))\subseteq f_{j-1}^{-1}(Q),\]
and Lemma \ref{eqq} shows that $f_{j-1}^{-1}(Q)\in\Imm(\varphi_j)$.
	
	This concludes the induction. The result follows by taking $j=n$.
\end{proof}

\section{Poisson prime quotients of $A$ and $\ov{A}$}
\label{sectionCj}

	In this section we study the behaviour of the Poisson prime quotients of a Poisson algebra $A \in \mathcal{P}$ under the deleting derivations algorithm. We continue using notation from Hypothesis \ref{hyp1} and Section \ref{ddap}. 
	
		Fix $2\leq j\leq n$, let $P\in\PS(C_{j+1})$ and set $Q:=\varphi_j(P)\in\PS(C_j)$. As usual, to simplify notation we set $U_i:= X_{i,j+1}$ and $V_i:= X_{i,j}$ for all $i$. We also set $D:=C_{j+1}/P$ and $E:=C_j/Q$. Finally, we set $d_i:=U_i+P$ and $e_i:=V_i+Q$ for all $1\leq i\leq n$.
	
\begin{lem}
\label{d=0}
If $d_j=0$, then there is a Poisson algebra isomorphism between $E$ and $D$ sending $e_i$ to $d_i$ for all $1\leq i\leq n$.
\end{lem}

\begin{proof}
$d_j=0$ means that $P\in\mathcal{P}_j^1(C_{j+1})$ and $Q=g_j^{-1}(P/\langle U_j\rangle_P)$. Thus we have a surjective Poisson algebra homomorphism:
	\[C_j\longrightarrow\frac{C_{j+1}/\langle U_j\rangle_P}{P/\langle U_j\rangle_P}\cong C_{j+1}/P,\]
whose kernel is $Q$.
\end{proof}

\begin{lem}
\label{quo}
Assume that $d_j\neq0$ and set $\ov{S_j}:=\{d_j^n\ |\ n\geq0\}$. Then there is an injective Poisson algebra homomorphism $\Lambda:E\rightarrow D\ov{S_j}^{-1}$ defined by:
\begin{align*}
	\Lambda(e_i)=\left\{
\begin{array}{ll}
d_i &\quad i\geq j,\\
\sum\limits_{k\geq0}\frac{1}{\eta_j^k}\ov{D_{j,k}(U_i)}d_j^{-k} &\quad  i<j
\end{array}
\right.
\end{align*}
where $\ov{D_{j,k}(U_i)}:=D_{j,k}(U_i)+P$.
\end{lem}

\begin{proof}
By assumption $P\in\mathcal{P}_j^0(C_{j+1})$, so $QS_j^{-1}=PS_j^{-1}$ is an ideal in $C_jS_j^{-1}=C_{j+1}S_j^{-1}$ and we have the following identifications:
\[\frac{C_jS_j^{-1}}{QS_j^{-1}}=\frac{C_{j+1}S_j^{-1}}{PS_j^{-1}}\cong D\ov{S_j}^{-1}.\]
Thus the canonical embedding of $C_j$ in $C_jS_j^{-1}$ induces a well-defined injective Poisson algebra homomorphism $\Lambda$ from $E$ to $D\ov{S_j}^{-1}$ whose expression is clear from the equalities:
\begin{align*}
	V_i=\left\{
\begin{array}{ll}
U_i &\quad i\geq j,\\
\sum\limits_{k\geq0}\frac{1}{\eta_j^k}D_{j,k}(U_i)U_j^{-k} &\quad  i<j.
\end{array}
\right.
\end{align*}
\end{proof}

From Lemma \ref{d=0} and Lemma \ref{quo}, we can state:

\begin{cor}
\label{fquo}
$D$ and $E$ have the same Poisson field of fractions (if $U_j\notin P$, we identify $E$ with its image in $DS_j^{-1}$ by $\Lambda$ so that we have $D\ov{S_j}^{-1}=E\ov{S_j}^{-1}$).
\end{cor}

An easy induction gives us the following result on the Poisson structure of the fields of fractions of the Poisson prime quotients of $A$. 

\begin{cor}
\label{quotients} Let $A\in\mathcal{P}$, $P\in\PS(A)$ and set $Q:=\varphi(P)\in\PS(\ov{A})$. Then we have a Poisson algebra isomorphism:
\[\F\big(A/P\big)\cong\F\big(\ov{A}/Q\big).\]
\end{cor}

In particular this corollary says that in order to prove the quadratic Poisson Gel'fand-Kirillov problem (see \cite{GL} or \cite{LL}) for the Poisson prime quotients of $A$ it is enough to prove it for the Poisson prime quotients of the Poisson affine space $\ov{A}$. We retrieve the result of \cite[Therorem 3.3 (2)]{LL} with the addition that the ideal $Q$ is now charaterised by the canonical embedding. In characteristic zero the Poisson prime quotients of a Poisson affine space indeed satisfy the quadratic Poisson Gel'fand-Kirillov problem (\cite[Theorem 3.3]{GL}), but this is not clear anymore in positive characteristic.

\section{Poisson Dixmier-Moeglin equivalence}
\label{PDM}

	In this section we prove that the \emph{Poisson Dixmier-Moeglin equivalence} holds for the Poisson algebras of the class $\mathcal{P}$ when $\car\K=0$. As stated in the introduction it only remains to show that the Poisson rational ideals of $A\in\mathcal{P}$ are also locally closed. We continue to use the notation of Hypothesis \ref{hyp1} and of Sections \ref{ddap} and \ref{sectionCj}. For a Poisson prime ideal $P$ of a Poisson algebra $A$ we set:
\[V(P)=\{I\in\PS(A)\ |\ I\supseteq P\}\quad \text{and}\quad 
W(P)=\{I\in\PS(A)\ |\ I\not\supseteq P\}.\]
The set $V(P)$ is a closed set of $\PS(A)$ and $W(P)$ is an open of $\PS(A)$. The following lemma is a Poisson version of \cite[Lemma II.7.7]{BG}.

\begin{lem}
\label{equu}
Let $A$ be a Poisson algebra and $P\in\PS(A)$. Then $P$ is locally closed if and only if the intersection of all the Poisson prime ideals properly containing $P$ is an ideal properly containing $P$.
\end{lem}

\begin{proof}
Let $\mathcal{I}$ be the intersection of all the Poisson prime ideals of $A$ properly containing $P$. If $P\varsubsetneq \mathcal{I}$, then $W(\mathcal{I})\cap V(P)=\{P\}$, i.e. $\{P\}$ is a locally closed point $\PS(A)$. Conversely, if $P$ is locally closed, then there are ideals $I$ and $L$ in $A$ such that $V(I)\cap W(L)=\{P\}$. Therefore we can see that $P \varsubsetneq L+P\subseteq \mathcal{I}$.
\end{proof}

Hence $P$ is locally closed if and only if the intersection of all non trivial Poisson prime ideals in $A/P$ is non trivial.

\begin{prop}
\label{pdmm}
Let $A\in\mathcal{P}$ and assume that $\car\K=0$. Then Poisson rational ideals of $A$ are Poisson locally closed ideals.
\end{prop}

\begin{proof}
Recall that by applying the Poisson deleting derivations algorithm to the Poisson algebra $A$ we get a sequence of Poisson algebras $C_j$ where $j$ runs from $n+1$ to $2$ such that $C_{n+1}=A$ and $C_2=\ov{A}$ is a Poisson affine space. We will show by an increasing induction on $j$ that all Poisson rational ideals of $C_j$ are locally closed. When $j=2$ the algebra $\ov{A}$ is a Poisson affine space and the result comes from \cite[Example 4.6]{Goo1}. Assume that for some $2\leq j\leq n$ the Poisson rational ideals of $C_j$ are locally closed. Let $P\in\PS(C_{j+1})$ be a Poisson rational ideal. We distinguish between two cases: either $U_j\in P$, or $U_j\notin P$.

	\underline{Case 1:} If $U_j\in P$, then by Lemma \ref{d=0} we get a Poisson algebra isomorphism between $C_{j+1}/P$ and $C_{j}/\varphi_j(P)$, and the result follows.

	\underline{Case 2:} If $U_j\notin P$, then by Lemma \ref{quo} we get the equality $C_jS_j^{-1}/ QS_j^{-1}=C_{j+1}S_j^{-1}/ PS_j^{-1}$, which leads to the isomorphism:
  \[Z_P\Big(\F\Big(\frac{C_{j+1}}{P}\Big)\Big)\cong Z_P\Big(\F\Big(\frac{C_{j}}{\varphi_j(P)}\Big)\Big).\]
Therefore $\varphi_j(P)\in\PS(C_j)$ is Poisson rational, and so is locally closed. We now introduce a few notation: 
\begin{align*}
&\mathcal{F}_{j}^0:=\{Q\in\PS(C_j)\ |\ \varphi_j(P)\varsubsetneq Q \text{ and } V_j \notin Q\},\\
&\mathcal{F}_{j}^1:=\{Q\in\PS(C_j)\ |\ \varphi_j(P)\varsubsetneq Q \text{ and } V_j \in Q\},\\
&\mathcal{F}_{j+1}^0:=\{Q\in\PS(C_{j+1})\ |\ P\varsubsetneq Q \text{ and } U_j \notin Q\},\\
&\mathcal{F}_{j+1}^1:=\{Q\in\PS(C_{j+1})\ |\ P\varsubsetneq Q \text{ and } U_j \in Q\},
\end{align*}
\[\mathcal{T}_{j}^0:=\bigcap_{Q\in\mathcal{F}_{j}^0} Q,\quad \mathcal{T}_{j}^1:=\bigcap_{Q\in\mathcal{F}_{j}^1} Q, \quad \mathcal{T}_{j+1}^0:=\bigcap_{Q\in\mathcal{F}_{j+1}^0} Q,\quad \text{and}\quad \mathcal{T}_{j+1}^1:=\bigcap_{Q\in\mathcal{F}_{j+1}^1} Q.\]
Let $\mathcal{I}$ be the intersection of all the Poisson prime ideals of $C_{j+1}$ properly containing $P$. We have:
\begin{align}
\label{gt}
\Big(P\text{ locally closed}\Big) \iff \Big(P\varsubsetneq\mathcal{I} \Big) \iff \Big(P\varsubsetneq\big(\mathcal{T}_{j+1}^0 \cap\mathcal{T}_{j+1}^1\big) \Big).
\end{align}
By the induction hypothesis we have:
\begin{align*}
\varphi_j (P) \varsubsetneq\big(\mathcal{T}_{j}^0 \cap\mathcal{T}_{j}^1\big) \quad\text{so that} \quad \varphi_j(P)=PS_j^{-1}\cap C_j\varsubsetneq \mathcal{T}_{j}^0.
\end{align*}
Since the map $\varphi_j$ restricts to a homeomorphism from $\mathcal{F}_{j+1}^0$ to $\mathcal{F}_{j}^0$ we have:
\[\varphi_j(P)\varsubsetneq \mathcal{T}_{j}^0 \iff P\varsubsetneq \mathcal{T}_{j+1}^0.\]
Therefore there exists $a\in\big(\mathcal{T}_{j+1}^0\setminus P\big)$. Moreover by definition we have $U_j\in\big(\mathcal{T}_{j+1}^1\setminus P\big)$. Since $P$ is a prime ideal and $a,U_j\notin P$ it clear that: 
\[aU_j\in\Big(\mathcal{T}_{j+1}^0 \cap\mathcal{T}_{j+1}^1 \setminus P\Big),\]
and by (\ref{gt}) we obtain that $P$ is locally closed. This concludes the induction. The case $j=n$ gives us the result for $C_{n+1}=A$.
\end{proof}

We are now ready to state the main results of this section.

\begin{thm}
Let $A\in\mathcal{P}$ and assume that $\car\K=0$. Then $A$ satisfies the Poisson Dixmier-Moeglin equivalence.
\end{thm}

\begin{cor}
\label{primprim}
Let $A\in\mathcal{P}$ and assume that $\car\K=0$. Then for all $P\in\PS(A)$ we have the following equivalence:
\[P\text{ is Poisson primitive in }A \iff \varphi(P)\text{ is Poisson primitive in }\ov{A}.\]
\end{cor}

We can also describe the primitive ideals of $A\in\mathcal{P}$ inside their stata, namely they are exactely the maximal ideals in their respective strata.

\begin{prop}
Let $A\in\mathcal{P}$ and assume that $\car\K=0$. Suppose that $w\in W'_P$ and let $P\in\PS_w(A)$. Then:
\[P\text{ is Poisson primitive }\iff\ P\text{ is maximal in }\PS_w(A).\]
\end{prop}

\begin{proof}
	First suppose that $P$ is a Poisson primitive ideal. Then $\varphi(P)\in\PS_w(\ov{A})$ is Poisson primitive in $\ov{A}$ by Corollary \ref{primprim}. By \cite[Theorem 4.3, Example 4.6]{Goo1}, $\varphi(P)$ is maximal in $\PS_w(\ov{A})$. Now let $P'\in\PS_w(A)$ be such that $P\subseteq P'$. Since $\varphi$ induces a homeomorphism from $\PS_w(A)$ to $\varphi(\PS_w(A))\subseteq \PS_w(\ov{A})$, we have $\varphi(P)\subseteq\varphi(P')$ inside $\PS_w(\ov{A})$. By maximality of $\varphi(P)$ we get $\varphi(P)=\varphi(P')$, i.e. $P=P'$, and $P$ is maximal in $\PS_w(A)$.

	Conversely, suppose that $P$ is maximal in $\PS_w(A)$. Then $\varphi(P)$ is maximal in $\varphi\big(\PS_w(A)\big)$ by Theorem \ref{topo}. Recall that $\varphi\big(\PS_w(A)\big)\subseteq\PS_w(\ov{A})$ by Theorem \ref{topo}, and let $Q\in\PS_w(\ov{A})$ such that $\varphi(P)\subseteq Q$. By Proposition \ref{wPQ} we have $Q\in\Imm(\varphi)$, i.e. $Q\in\varphi\big(\PS_w(A)\big)$ and by maximality of $\varphi(P)$ in $\varphi\big(\PS_w(A)\big)$ we have $Q=\varphi(P)$. Therefore $\varphi(P)$ is maximal in $\PS_w(\ov{A})$. By \cite[Theorem 4.3, Example 4.6]{Goo1} this shows that $\varphi(P)$ is Poisson primitive in $\ov{A}$. We conclude by Corollary \ref{primprim} that $P$ is Poisson primitive in $A$.
\end{proof}

	In characteristic zero every iterative, locally nilpotent Poisson $\al$-derivation such that $[\de,\al]=\eta\de$ for some nonzero scalar $\eta$, extends to an iterative, locally nilpotent higher $(\eta,\al)$-skew Poisson derivation, so that Hypothesis \ref{hyp1} is easier to check in that case.

	We have the following transfer result, which can be proved in a similar way as Proposition \ref{pdmm}, thanks to Theorem \ref{iso}.

\begin{thm}
\label{transfer}
Assume that $\car\K=0$. Let $A$ be an affine Poisson $\K$-algebra, $\al\in\D_P(A)$ and $\de$ be a locally nilpotent Poisson $\al$-derivation such that $[\de,\al]=\eta\de$ for some nonzero scalar $\eta$. If the Poisson-Ore extension $A[X;\al]_P$ satisfies the Poisson Dixmier-Moeglin equivalence, then the Poisson-Ore extension $A[X;\al,\de]_P$ satisfies the Poisson Dixmier-Moeglin equivalence.
\end{thm}

\begin{ex}
\label{examplePDME}
The algebra $A=B[X_4;\al,\de]_P$ of Example \ref{ex4} satisfies the Poisson Dixmier-Moeglin equivalence. Indeed, the Poisson algebra $B[X_4;\al]_P$ is a Poisson affine space and thus satisfies the Poisson Dixmier-Moeglin equivalence (\cite[Example 4.6]{Goo1}). Moreover $[\de,\al]=\de$ and $\de$ is locally nilpotent, so we can apply Theorem \ref{transfer}. Note that the torus $H:=(\mathbb{K}^{\ast})^2$ acts by Poisson automorphisms on this algebra via:
$$h\cdot X_1=h_1 X_1,\quad h\cdot X_2=h_1 X_2,\quad h\cdot X_3=h_2 X_3,\quad \mbox{and}\quad h\cdot X_4=h_1h_2^{-1} X_4, $$
for all $h=(h_1,h_2) \in H$. However, the fact that $A$ satisfies the Poisson Dixmier-Moeglin equivalence cannot be deduced from \cite[Theorem 4.3]{Goo1} with this natural torus action as $A$ has infinitely many Poisson $H$-invariant prime ideals (it is easy to check that, for all $\lambda \in \mathbb{K}$, the ideal generated by $X_1+\lambda X_2$ is a Poisson $H$-invariant prime ideal).
\end{ex}

\section{Quantum and Poisson matrices: toward a homeomorphism between spectrum and Poisson spectrum}
\label{PM}

In this section we assume that $\car\K=0$ and that $q\in\K^{\ast}$ is not a root of unity. It is conjectured in \cite{Goo2} that, among other quantised coordinate rings, the spectrum of the algebra of quantum matrices is homeomorphic to the Poisson spectrum of its semiclassical limit. In this section we present a step toward proving this conjecture. 
The single parameter coordinate ring of quantum matrices is denoted by $R:=\mathcal{O}_q\big(M_{m,p}(\K)\big)$ (see \cite[Section I.2.2]{BG} for a definition). Its semiclassical limit, denoted by $A$, is the polynomial algebra $\K[X_{ij}\ |\ 1\leq i\leq m,\ 1\leq j\leq p\ ]$ endowed with the Poisson bracket:
\begin{equation*}
 \{X_{ij},X_{kl}\}= \left\{
		\begin{array}{lll}
			X_{ij}X_{kl} \ \ \ & \mbox{if } i<k \mbox{ and } j=l,\\
			X_{ij}X_{kl} & \mbox{if } i=k \mbox{ and } j<l,\\
			0 & \mbox{if } i<k \mbox{ and } j>l,\\
			2X_{il}X_{kj} & \mbox{if } i<k \mbox{ and } j<l.
		\end{array}	
\right.
\end{equation*}
For more details on the semiclassical limit process see \cite[Section 2]{Goo2}. Set $W=\mathcal{P}\big([\mspace{-2 mu} [1,m] \mspace{-2 mu} ]\times [\mspace{-2 mu} [1,p] \mspace{-2 mu} ]\big)$. Thanks to Cauchon's deleting derivations algorithm (see \cite{Cau1}), the spectrum $\Sp(R)$ of $R$ is partitioned into strata, denoted by $\Sp_w(R)$, indexed by the elements of a subset $W'$ of $W$. It is shown in \cite[Section 7.3]{myphd} that the Poisson algebra $A$ belongs to the class $\mathcal{P}$, so that we can perform the Poisson deleting derivations algorithm, and that the set of Cauchon diagrams $W_P'$ coincides with $W'$.

We now compare the strata $\Sp_w(R)$ and $\PS_w(A)$ associated to the same $w\in W'=W_P'$. We will need the following observation. The algebra $\ov{R}$, obtained at the end of Cauchon's deleting derivations algorithm, is a quantum affine space associated to a multiplicatively skew-symmetric matrix $\bo{q}:=(q_{(i,j),(u,v)})$ of the form $q_{(i,j),(u,v)}=q^{\la_{(i,j),(u,v)}}$ for some skew-symmetric matrix $\bo\la:=(\la_{(i,j),(u,v)})$ (the matrix $\bo\la$ is made explicit in \cite[Section 4.1]{BL} for instance). It is a direct consequence of the semiclassical limit process that $\bo\la$ is the matrix defining the Poisson affine space $\ov{A}$, obtained at the end of the Poisson deleting derivations algorithm.

\begin{prop}
\label{www}
Let $w\in W'=W_P'$. Then there is a homeomorphism between $\PS_w(A)$ and $\Sp_w(R)$. More precisely we have:
\[\PS_w(A)\cong\Sp(\K[U_1^{\pm1},\dots,U_s^{\pm1}])\cong \Sp_w(R),\]
where $s$ is equal to the dimension over $\Q$ of the kernel of a matrix $M(w)$, obtained from the matrix $\bo\la$ by deleting rows and columns indexed by $(i,j)\in w$. 
\end{prop}

\begin{proof}
The homeomorphism $\Sp_w(R)\cong\Sp(\K[U_1^{\pm1},\dots,U_s^{\pm1}])$ follows from \cite[Theorem 3.1]{BL} and the observation made before the proposition. 

	To prove the homeomorphism $\PS_w(A)\cong\Sp(\K[U_1^{\pm1},\dots,U_s^{\pm1}])$ we proceed as follows. From Theorem \ref{topo} and \cite[Theorem 7.3.8]{myphd} the stratum $\PS_w(A)$ is homeomorphic to the stratum $\PS_w(\ov{A})$ via the canonical embedding. Recall that $\ov{A}$ is the Poisson affine space $\K_{\bo\la}[T_{11},\dots,T_{mp}]$. We denote by $J_w$ the Poisson ideal of $\ov{A}$ generated by the $T_{ij}$ for $(i,j)\in w$, and by $S_w$ is the multiplicative set of $\ov{A}/J_w$ generated by the image of the $T_{ij}$ for $i\in \ov{w}:=\big([\mspace{-2 mu} [1,m] \mspace{-2 mu} ]\times [\mspace{-2 mu} [1,p] \mspace{-2 mu} ]\big)\setminus w$. It results from the definition of $\PS_w(\ov{A})$ (see Section \ref{injec}) that there is a homeomorphism between $\PS_w(\ov{A})$ and $\PS(T)$, where $T=(\ov{A}/J_w)S_w^{-1}$ is the Poisson torus associated to $M(w)$. By \cite[Lemma 1.2]{Van}, a Poisson ideal of a Poisson torus is generated by its intersection with the Poisson centre, thus:
\[\PS_w(A)\cong\PS_w(\ov{A})\cong\PS(T)\cong \PS(Z_P(T)).\] 
By \cite[Lemma 1.2]{Van}, the Poisson centre of $T$ is the group algebra of the free abelian group:
	\[S:=\big\{\al\in\Z^r\ |\ \al M(w)\be^{\text{tr}}=0\ \text{ for all }\be\in\Z^r\big\},\]
where $r$ is the cardinality of $\ov{w}$ and the elements of $\Z^r$ are seen as row vectors. To conclude we remark that a basis of $S$ has the same cardinality as a basis of the kernel of the matrix $M(w)$.
\end{proof}

To summarise, we have just proved Theorem \ref{thm-section5}, i.e. there is a bijection between $\Sp(R)$ and $\PS(A)$ which induces by restriction homeomorphisms from $\Sp_w(R)$ and $\PS_w(A)$ for all $w\in W'=W_P'$. However it is unclear whether this bijection is a homeomorphism or not. The main obstruction is that the canonical embedding is only continuous on strata.



\bibliography{biblio}

\bibliographystyle{amsplain}

\end{document}